\documentclass{amsart}

 \usepackage{graphicx,color}
\usepackage{geometry}
\usepackage{calrsfs}

\def\A{\mathbf{A}}
\def\F{\mathbf{F}}

\def\GL{\mathop{\mathrm{GL}}\nolimits}

\def\Aut{\mathop{\mathrm{Aut}}\nolimits}
\def\cA{\mathcal{A}}
\def\cT{\mathcal{T}}
\def\cV{\mathsf{T}}
\def\cS{\mathcal{S}}
\def\diag{\mathop{\mathrm{diag}}\nolimits}
\DeclareMathOperator{\psl}{\![\![\!}
\DeclareMathOperator{\psr}{\!]\!]}
\renewcommand{\tilde}{\widetilde}
\def\t{\mathfrak t}

\newtheorem{theorem}{Theorem}[section]
\newtheorem{lemma}[theorem]{Lemma}
\newtheorem{proposition}[theorem]{Proposition}
\newtheorem{corollary}[theorem]{Corollary}
\theoremstyle{definition}
\newtheorem{remark}[theorem]{Remark}
\newtheorem{example}[theorem]{Example}

\usepackage{times}

\begin{document}

\date{} 
\title[toroidal automorphic forms]{Toroidal automorphic forms for some function fields}
\author[G.\ Cornelissen]{Gunther Cornelissen}
\address{Mathematisch Instituut, Universiteit Utrecht, Postbus 80.010, 3508 TA Utrecht, Nederland}
\email{\{cornelis,lorschei\}@math.uu.nl}
\author[O.\ Lorscheid]{Oliver Lorscheid} 
%\address{Mathematisch Instituut, Universiteit Utrecht, Postbus 80.010, 3508 TA Utrecht, Nederland}
%\email{lorschei@math.uu.nl}
%\subjclass[2000]{}

\maketitle

\begin{abstract} 
\noindent 
Zagier introduced toroidal automorphic forms to study the zeros of zeta functions: an automorphic form on $\GL_2$ is toroidal if all its right translates integrate to zero over all nonsplit tori in $GL_2$, and an Eisenstein series is toroidal if its weight is a zero of the zeta function of the corresponding field. We compute the space of such forms for the global function fields of class number one and genus $ g \leq 1$, and with a rational place. The space has dimension $g$ and is spanned by the expected Eisenstein series. We deduce an ``automorphic'' proof for the Riemann hypothesis for the zeta function of those curves.

\end{abstract}
\section{Introduction}

Let $X$ denote a smooth projective curve over a finite field $\F_q$ with $q$ elements, $\A$ the adeles over its function field $F:=\F_q(X)$, $G=\GL_2$, $B$ its standard (upper-triangular) Borel subgroup, $K=G({\mathcal O}_\A)$  the standard maximal compact subgroup of $G_{\A}$, with ${\mathcal O}_\A$ the maximal compact subring of $\A$, and $Z$ the center of $G$. Let $\cA$ denote the space of unramified automorphic forms $f: {G_F}\backslash G_\A/KZ_\A \to \bf C$. We use the following notations for matrices:  $$\diag(a,b)=\big( \begin{smallmatrix} a & 0 \\ 0 & b \end{smallmatrix} \big) \mbox{ and } \psl a,b \psr=\big( \begin{smallmatrix} a & b \\ 0 & 1 \end{smallmatrix} \big).$$

There is a bijection between quadratic separable field extensions $E/F$ and conjugacy classes of maximal non-split tori in $G_F$ via $$ E^\times = \Aut_E(E)  \subset  \Aut_F(E)  \simeq  G_F. $$  If $T$  is a non-split torus in $G$ with $T_F \cong E^\times$, define the space of \emph{toroidal automorphic forms for $F$ with respect to $T$ (or $E$)} to be
 \begin{equation} \label{def} \cV_F(E) = \left\{ \ f \, \in \, \cA \right. \ \mid \ \forall g \, \in \, G_\A, 
   \int\limits_{T_F Z_\A\backslash T_\A} \hspace*{-1em} f(tg) \, dt = 0 \left.\right\}. \end{equation} The integral makes sense since $T_F Z_\A\backslash T_\A$ is compact, and the space only depends on $E$, viz., the conjugacy class of $T$. The space of \emph{toroidal automorphic forms for $F$} is 
$$\cV_F = \bigcap\limits_E \, \cV_F(E), $$ where the intersection is over all quadratic separable $E/F$. The interest in  these spaces lies in the following version of a formula of Hecke (\cite{Hecke}, Werke p.\ 201); see Zagier, \cite{Zagier} pp.\ 298--299 for this formulation, in which the result essentially follows from Tate's thesis:

\begin{proposition} \label{Hecke} Let $\zeta_E$ denote the zeta function of the field $E$. Let $\varphi: \A^2 \rightarrow {\mathbf C}$  be a Schwartz-Bruhat function. Set $$f(g,s)= |\det g|^s_F \int\limits_{\A^\times} \varphi((0 , a)g)|a|^{2s} d^\times a .$$ 
An Eisenstein series $E(s)$
$$E(s)(g):= \sum\limits_{\gamma \in B_F \backslash G_F} f(\gamma g, s) \ \ \ \ (\mathrm{Re}(s)>1)$$ satisfies
$$\int\limits_{T_F Z_\A\backslash T_\A} \hspace*{-1em} E(s)(tg) \, dt = c(\varphi,g,s) |\det g|^s \zeta_E(s) $$ for some holomorphic function $c(\varphi,g,s)$. For every $g$ and $s$, there exists a function $\varphi$ such that $c(\varphi,g,s) \neq 0$. In particular, $E(s) \in \cV_F(E) \iff \zeta_E(s)=0$. \hfill $\Box$
\end{proposition}

\begin{remark} Toroidal integrals of parabolic forms are ubiquitous in the work of Waldspurger (\cite{Wald}, for recent applications, see Clozel and Ullmo \cite{CU} and Lysenko \cite{Lys}). Wie\-lon\-sky and Lachaud studied analogues for $\GL_n, \ n \geq 2$, and tied up the spaces with Connes' view on zeta functions (\cite{Wielonsky}, \cite{Lachaud}, \cite{Lachaud1}, \cite{Connes}). 
\end{remark} 

Let $\mathcal H = C^\infty_0(K \backslash G_\A / K)$ denote the bi-$K$-invariant Hecke algebra, acting by convolution on $\cA$. There is a  correspondence between $K$-invariant $G_\A$-modules and Hecke modules; in particular, we have
\begin{lemma}  The spaces $\cV_F(E)$ (for each $E$ with corresponding torus $T$) and $\cV_F$ are invariant under the Hecke algebra $\mathcal H$, and  \begin{equation} \label{defH} \cV_F(E) \subseteq \left\{ \ f \, \in \, \cA \right. \ \mid \ \forall \Phi \, \in \, {\mathcal H}, 
   \int\limits_{T_F Z_\A\backslash T_\A} \hspace*{-1em} \Phi(f)(t) \, dt = 0 \left.\right\}. \qed\end{equation} \end{lemma}

Now assume $F$ has class number one and there exists a place $\infty$ of degree one for $F$; let $t$ denote a local uniformizer at $\infty$. Strong approximation implies that we have a bijection
$${G_F}\backslash G_{\bf A}/K Z_\infty \stackrel{\sim}{\longrightarrow}  \Gamma \ \backslash G_\infty / K_\infty Z_\infty,$$
where $\Gamma=G(A)$ with $A$ the ring of functions in $F$ holomorphic outside $\infty$, and a subscript $\infty$ refers to the $\infty$-component. We define a graph $\cT$ with vertices $V \mathcal T = G_\infty\, /\, K_\infty \, Z_\infty$. If $\sim$ denotes equivalence of matrices modulo $K_\infty \, Z_\infty$,  then we call vertices in $V\mathcal T$ given by classes represented by matrices $g_1$ and $g_2$  adjacent, if 
$g_1^{-1}\, g_2 \sim \psl t, b \psr \ {\rm or} \ \psl  t^{-1}, 0 \psr$ for some $b \in {\mathcal O}_\infty/t$. Then $\cT$ is a tree that only depends on $q$ (the so-called Bruhat-Tits tree of $\mathrm{PGL}(2,F_\infty)$, cf.\ \cite{S1}, Ch.\ II). 

 The Hecke operator $\Phi_\infty \in {\mathcal H}$ given by the characteristic function of $K \psl t, 0 \psr K $ maps a vertex of $\cT$  to its neighbouring vertices. The action of $\Phi_\infty$ on the quotient graph $\Gamma \backslash \cT$ can be computed from the orders of the $\Gamma$-stabilizers of vertices and edges in $\cT$. When drawing a picture of $\Gamma \backslash \cT$, we agree to label a vertex along the edge towards an adjacent vertex by the corresponding weight of a Hecke operator, as in the next example.
 
 \begin{example} In Figure \ref{fig1}, one sees the graph $\Gamma \backslash \cT$ for the function field of $X={\bf P}^1$, with the well-known vertices representing $\{c_i =  \psl \pi^{-i}, 0 \psr \}_{i \geq 0}$  and the weights of $\Phi_\infty$, meaning \begin{equation}\label{HeckeP1} \mbox{for }n\geq 1,\ \Phi_\infty (f)(c_n) = q f(c_{n-1}) + f(c_{n+1}) \mbox{ and } \Phi_\infty (f)(c_0) = (q+1) f(c_1). \end{equation}
 \end{example}
 
 \begin{figure}[ht]
  \centering
   \input{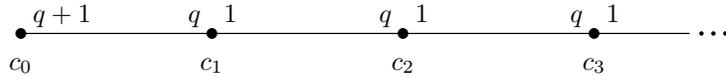} 
  \caption{The graph $\Gamma \backslash \cT$ for $X={\bf P}^1$}
   \label{fig1}
\end{figure}

\section{The rational function field}

First, assume $X={\mathbf P}^1$ over ${\mathbf F}_q$, so $F$ is a rational function field. Set $E={\mathbf F}_{q^2} F$ the quadratic constant extension of $F$.

\begin{theorem} \label{P1} $\cV_{F}=\cV_{F}(E)=\{0\}.$ \end{theorem}

\begin{proof}
Let $T$ be a torus with $T_F=E^\times$, that has a basis over $F$ contained in the constant extension ${\bf F}_{q^2}$. The integral defining $f \in \cV_F(E)$ in equation (\ref{def}) for the element $g=1 \in G_\A$ becomes $$  \int\limits_{T_F Z_\A\backslash T_\A} \hspace*{-1em}  f(t) \, dt \ = \ \kappa \cdot \hspace*{-2em} \int\limits_{T_F Z_\A\backslash T_\A/(T_{\A} \cap K)} \hspace*{-2em} f(t) \, dt \ = \ \kappa \cdot \hspace*{-1em} \int\limits_{E^\times \A_F^\times \backslash \A_E^\times / {\mathcal O}_{\A_E}^\times} \hspace*{-2em}f(t) \, dt \ = \ \kappa \cdot f(c_0), $$
with $\kappa=\mu(T_\A \cap K) \neq 0$.  Indeed, by our choice of ``constant'' basis, we have $T_\A \cap K \cong {\mathcal O}_{\A_E}^\times$. For the final equality, note that the integration domain $E^\times \A_F^\times \backslash \A_E^\times / {\mathcal O}_{\A_E}^\times$ is isomorphic to the quotient of the class group of $E$ by that of $F$, and that both of these groups are trivial, so map to the identity matrix $c_0$ in $\Gamma \backslash \cT$. 

Hence we first of all find $f(c_0)=0$. For $\Phi=\Phi_\infty^k$, this equation transforms into $(\Phi_\infty^k f)(c_0)=0$ (cf.\ (\ref{defH})), and with (\ref{HeckeP1}) this leads to a system of equations for $f(c_i) \ (i=1,2,\dots)$ that can easily be shown inductively to only have the zero solution $f=0$. \end{proof}

\section{Three elliptic curves}

Now assume that $F$ is not rational, has class number one, a rational point $\infty$ and genus $\leq 1$. In this paper, we focus on such fields $F$, since it turns out that the space $\cV_F$ can be understood elaborating only existing structure results about the graph $\Gamma \backslash \cT$. 

The Hasse-Weil theorem implies that there are only three possibilities for $F$, which we conveniently number as follows:  $\{F_q\}_{q=2}^4$ with $F_q$ the function field of the projective curve $X_q/{\bf F}_q \ (q=2,3,4)$ are the respective elliptic curves $$y^2+y=x^3+x+1, \ y^2=x^3-x-1 \mbox{ and } y^2+y=x^3+\alpha$$ with $\F_4=\F_2(\alpha)$. Let $F^{(2)}_q = \F_{q^2} F_q$ denote the quadratic constant extension of $F_q$. 

\begin{figure}[ht]
  \centering
   \input{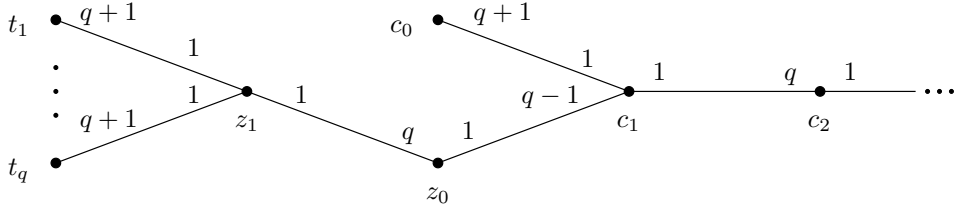} 
  \caption{The graph $\Gamma \backslash \cT$ for $F_q\ (q=2,3,4)$}
   \label{fig2}
\end{figure}

The graph $\Gamma \backslash \cT$ for $F_q \ (q=2,3,4)$ with the $\Phi_\infty$-weights is displayed in Figure \ref{fig2}, cf.\ Serre \cite{S1}, 2.4.4 and Ex.\ 3b)+3c) on page 117 and/or Takahashi \cite{Tak} for these facts. 

Further useful facts: One easily calculates that $X_q({\bf F}_{q^2})$ is cyclic of order $2q+1$; let $Q$ denote any generator. We will use lateron that the vertices $t_i$ correspond to classes of rank-two vector bundles on $X_q({\bf F}_q)$ that are pushed down from line bundles on $X_q({\bf F}_{q^2})$ given by multiples $Q,2Q,\dots,qQ$ of   $Q$, cf.\ Serre, loc.\ cit. For a representation in terms of matrices, one may refer to \cite{Tak}: if $iQ=(\ell,\ast) \in X_q({\bf F}_{q^2})$, then $t_i=\psl t^2, t^{-1}+\ell t \psr$.

\noindent We denote a function $f$ on $\Gamma \backslash \cT$ by a vector  $$f=[f(t_1),\dots,f(t_q) \mid f(z_0),f(z_1) \mid f(c_0),f(c_1),f(c_2),\dots].$$

\begin{proposition} \label{prop} A function $f \in \cV_{F_q}(F^{(2)}_q) \ (q=2,3,4)$ belongs to the $\Phi_\infty$-stable linear space $\cS$ of functions  \begin{equation} \label{basis} \cS := \{ \, [T_1,\dots,T_q \mid Z_0,Z_1 \mid C_0,C_1,C_2, \dots] \, \} \end{equation} with $C_0=-2(T_1+\dots+T_q)$ and for $k \geq 0$, \begin{equation}\label{blah} C_k = \left\{ \begin{array}{ll} \lambda_k Z_0 + \mu_k (T_1+\dots+T_q) & \mbox{ if } $k$ \mbox{ even} \\ \nu_k Z_1 & \mbox{ if } $k$ \mbox{ odd} \end{array} \right. \end{equation} 
for some constants $ \lambda_k, \mu_k, \nu_k$. In particular, $$\dim \cV_{F_q}(F^{(2)}_q) \leq \dim \cS =  q+2, $$ and $\dim \cV_{F_q}$ is finite.
\end{proposition}

\begin{proof} We choose arbitrary values $T_j$ at $t_j \ (j=1,\dots,q)$ and $Z_j$ at $z_j \ (j=1,2)$, and set $\tau=T_1+\dots+T_q$. We have $$ \int\limits_{T_F Z_\A\backslash T_\A} \hspace*{-5mm} f(t) \, dt = C_0+2 \tau.$$ Indeed, by the same reasoning as in the proof of Theorem \ref{P1}, the integration area maps to the image of $$\mathrm{ Pic}(X_q({\bf F}_{q^2}))/\mathrm{Pic}(X_q({\bf F}_q)) = X_q({\bf F}_{q^2})/X_q({\bf F}_{q}) = X_q({\bf F}_{q^2})$$ (the final equality since $X_q$ is assumed to have class number one) in $\Gamma \backslash \cT$, and these are exactly the vertices $c_0$ and $t_j$ (the latter with multiplicity two, since $\pm Q \in E({\bf F}_{q^2})$ map to the same vertex). The integral is zero exactly if $C_0=-2\tau$. Applying the Hecke operator $\Phi_\infty$ to this equation (cf.\ (\ref{defH})) gives $C_1=-2Z_1$, then applying $\Phi_\infty$ again gives $C_2=-(q+1)Z_0$. The rest follows by induction. If we apply $\Phi_\infty$ to the equations (\ref{blah}) for $k \geq 2$, we find by induction for $k$ even that $$C_{k+1} = \lambda_k C_1 + (\lambda_kq+\mu_kq(q+1)-q\nu_{k-1})Z_1$$ and for $k$ odd that $$C_{k+1} =(\nu_k-q\lambda_{k-1})Z_0+ (\nu_k-q\mu_{k-1})\tau.$$ \end{proof}

\begin{lemma} \label{basisofeig}
The space $\cS$ from \textup{(\ref{basis})} has a basis of $q+2$ $\Phi_\infty$-eigenforms, of which exactly $q-1$ are cusp forms with eigenvalue zero and support in the set of vertices $\{t_j\}$, and three are non-cuspidal forms with respective eigenvalues $0,q,-q$.
\end{lemma} 

\begin{proof} With $\tau=T_1+\dots+T_q$, the function $$f=  [T_1,\dots,T_q\mid Z_0,Z_1 \mid -2\tau,C_1,C_2, \dots]$$ is a $\Phi_\infty$-eigenform with eigenvalue $\lambda$ if and only if 
$$ \lambda T_j = (q+1) Z_1; \ \lambda Z_1 = \tau+Z_0; \ \lambda Z_0 = qZ_1+C_1; \ \lambda (-2\tau) = (q+1) C_1;  \mbox{ etc.} $$ We consider two cases:

 (a) if $\lambda=0$, we find $q$ forms $$f_k=[{0,\dots,0,1,0,\dots,0}\mid 0,-1 \mid -q,\dots]$$ with $T_j=1 \iff j=k$.

 (b) if $\lambda \neq 0$, we find $\lambda=\pm q$ with eigenforms $$f_{\pm} = [q+1,\dots,q+1 \mid -q,\pm q \mid -2q(q+1), \mp 2q^2, \dots].$$ Since we found $q+2$ eigenforms, they span $\cS$. From the fact that a cusp form satisfies $f(c_i)=0$ for all $i$ sufficiently large (cf.\ Harder \cite{Harder}, Thm.\ 1.2.1), one easily deduces that a basis of cusp forms in $\cS$ consists of $f_k-f_1$ for $k=2,\dots,q$. \end{proof}

\begin{corollary}
The Riemann hypothesis is true for $\zeta_{F_q}\ (q=2,3,4)$.
\end{corollary} 

\begin{proof} From Lemma \ref{basisofeig}, we deduce that the only possible $\Phi_\infty$-eigenvalue of a toroidal Eisenstein series is $\pm q$ or $0$, but on the other hand, from Lemma \ref{Hecke}, we know this eigenvalue is $q^s+q^{1-s}$ where $\zeta_{F_q}(s)=0$. We deduce easily that $s$ has real part $1/2$. \end{proof} 

\begin{remark}
One may verify that this proves the Riemann Hypothesis for the fields $F_q$ without actually computing $\zeta_{F_q}$: it only uses the expression for the zeta function by a Tate integral. Using a sledgehammer to crack a nut, one may equally deduce from Theorem \ref{P1} that $\zeta_{{\bf P}_1}$ doesn't have any zeros. At least the above corollary shows how enough knowledge about the space of toroidal automorphic forms does allow one to deduce a Riemann Hypothesis, in line with a hope expressed by Zagier \cite{Zagier}.
\end{remark}

\begin{theorem}
For $q=2,3,4$, $\cV_{F_q}$ is one-dimensional, spanned by the Eisenstein series of weight $s$  equal to a zero of the zeta function $\zeta_{F_q}$ of $F_q$.
\end{theorem}

\begin{remark}
Note that the functional equation for $E(s)$ implies  that $E(s)$ and $E(1-s)$ are linearly dependent, so it doesn't matter which zero of $\zeta_{F_q}$ is taken.
\end{remark} 

\begin{proof} By Lemma \ref{basisofeig},  $\cV_{F_q}$ is a $\Phi_\infty$-stable subspace of the  finite dimensional space $\cS$, and $\Phi_\infty$ is diagonalizable on $\cS$. By linear algebra, the restriction of $\Phi_\infty$ is also diagonalizable on $\cV_{F_q}$ with a subset of the given eigenvalues, hence $\cV_{F_q}$ is a subspace of the space of automorphic forms for the corresponding eigenvalues of $\Phi_\infty$. By \cite{Li}, Theorem 7.1, it can therefore be split into a direct sum of a space of Eisenstein series $\mathcal E$, a space of residues of Eisenstein series $\mathcal R$, and a space of cusp forms $\mathcal C$ (note that in the slightly different notations of \cite{Li}, ``residues of Eisenstein series'' are called ``Eisenstein series'', too).  We treat these spaces separately. 

{$\mathcal E:$}\ By Proposition \ref{Hecke}, $\cV_{F_q}(F_q^{(2)})$ contains exactly two Eisenstein series, one corresponding to a zero $s_0$ of $\zeta_{F_q}$, and one corresponding to a zero $s_1$ of $$L_q(s):=\zeta_{F_q^{(2)}}(s)/\zeta_{F_q}(s).$$ Now consider the torus $\tilde{T}$ corresponding to the quadratic extension $E_q=F_q(z)/F_q$ of genus two defined by $x=z(z+1).$ Set $$\tilde{L}_q(s):=\zeta_{E_q}(s)/\zeta_{F_q}(s)$$ and $T=q^{-s}$. One computes immediately that $L_q=qT^2+qT+1$ but $$\tilde{L}_2=2T^2+1, \tilde{L}_3=3T^2+T+1 \mbox{ and } \tilde{L}_4=4T^2+1.$$ Since $L_q$ and $\tilde{L}_q$ have no common zero, the $\tilde{T}$-integral of the Eisenstein series of weight $s_1$ is non-zero, and hence it doesn't belong to $\cV_{F_q}$. Hence $\mathcal E$ is as expected.

{$\mathcal R:$}\  Elements in $\mathcal R$ have $\Phi_\infty$-eigenvalues $\neq 0, \pm q$, so cannot even occur in $\cS$: since the class number of $F_q$ is one, $\mathcal R$ is spanned by the two forms $$r_\pm:=[1,\dots,1\mid \pm1,1 \mid 1,\pm 1,1,\pm1,\dots]$$ with $r(c_{i})=(\pm 1)^i$, and this is a $\Phi_\infty$-eigenform with eigenvalue $\pm(q+1)$. (In general, the space is spanned by elements of the form $\chi \circ \det$ with $\chi$ a class group character, cf.\ \cite{Gelbart}, p.\ 174.)

{$\mathcal C:$}\  By multiplicity one, $\mathcal C$ has a basis of simultaneous $\mathcal H$-eigenforms. From Lemma \ref{basisofeig}, we know that potential cusp forms in $\cV_{F_q}$ have support in the set of vertices $\{t_i\}$. To prove that ${\mathcal C}=\{0\}$, the following therefore suffices:

\begin{proposition} The only cusp form which is a simultaneous eigenform for the Hecke algebra $\mathcal H$ and has support in $\{t_i\}$ is $f = 0$. 
\end{proposition} 

{\it Proof.\/} Let $f$ denote such a form. Fix a vertex $\t \in \{t_i\}$. It corresponds to a point $P=(\ell, \ast)$ on $X_q({\bf F}_{q^2})$, which is a place of degree two of ${\bf F}_q(X_q)$. Let $\Phi_P$ denote the corresponding Hecke operator. We claim that 

\begin{lemma} \label{PHIP} $\Phi_P(c_0) = (q+1) c_2 + q(q-1) \t.$ \end{lemma} 

Given this claim, we finish the proof as follows: we assume that $f$ is a $\Phi_P$-eigenform with eigenvalue $\lambda_P$. Then 
$$ 0 = \lambda_P f(c_0) = \Phi_P f(c_0) = q(q-1) f(\t) + (q+1) f(c_2) = q(q-1)f(\t) $$ since $f(c_2)=0$,
hence $f(\t)=0$ for all $\t$.

\medskip 

{\it Proof of Lemma \ref{PHIP}\/} \  As in \cite{Gelbart}, 3.7, the Hecke operator $\Phi_P$  maps the identity matrix (= the vertex $c_0$) to the set of vertices corresponding to the matrices $ m_\infty:=\diag(\pi,1)$ and $m_b:=\big( \begin{smallmatrix} 1 & b \\ 0 & \pi \end{smallmatrix} \big)$, where $\pi=x-\ell$ is a local uniformizer at $P$ and $b$ runs through the residue field at $P$, which is $${\bf F}_q[X_q]/(x-\ell) = {\bf F}_q[y]/F(\ell,y) \cong {\bf F}_{q^2}$$ if $F(x,y)=0$ is the defining equation for $X_q$. Hence we can represent every such $b$ as $b=b_0+b_1 y$. We now reduce these matrices to a standard form in $\Gamma \backslash \cT$ from \cite{Tak}, \S 2. By right multiplication with $\psl 1,-b_0 \psr $, we are reduced to considering only $b=b_1 y$.

If $b_1=0$, then the matrix is $m_b=\diag(1,\pi)\sim \diag(\pi^{-1},1)$. Recall that $t=x/y$ is a uniformizer at $\infty$, so $x-\ell = t^{-2} \cdot A$ for some $A \in {\bf F}_q\psl t \psr^*$. Hence right multiplication by $\diag(A^{-1},1)$ gives that this matrix reduces to $c_2$. The same is true for $m_\infty$.

 On the other hand, if $b_1 \neq 0$, multiplication on the left by $\diag(1,b_1)$ and on the right by $\diag(1,b_1^{-1})$ reduces us to considering $m_y$. By multiplication on the right with $$\diag((x-\ell)^{-1}\cdot A,(x-\ell)^{-1}),$$ we get $m_y \sim \psl t^2, y/(x-\ell) \psr$. Now note that
$$\frac{y}{x-\ell} = \frac{y}{x} \cdot \left(1+ \frac{\ell}{x} + \left(\frac{\ell}{x}\right)^2 + \dots\right) = t^{-1}+\ell t + \beta(t) t^2$$ for some $\beta \in {\bf F}_q \psl t \psr $. Hence right multiplication with $\psl 1,-\beta \psr $ gives $m_y \sim \psl t^2,t^{-1} + \ell t \psr,$ and this is exactly the vertex $\t$. \end{proof}

\begin{remark} 
Using different methods, more akin the geometrical Langlands programme, the second author (\cite{Lor}) has generalized the above results as follows. For a general function field $F$ of genus $g$ and class number $h$, one may show that $\cV_F$ is finite dimensional. Its Eisenstein part is of dimension at least $h(g-1)+1$. Residues of Eisenstein series are never toroidal. For general elliptic function fields, there are no toroidal cusp forms. For a general function field, the analogue of a result of Waldspurger (\cite{Wald}, Prop.\ 7) implies that the cusp forms in $\cV_F$ are exactly those having vanishing central $L$-value. 
\end{remark}

\end{document}